\documentstyle[12pt]{article} \vsize=28.7truecm \hsize=23truecm \columnsep=0.8truecm

 \def\vbar{\mathchoice{\vrule height2.3ptdepth-.3ptwidth.12pt\kern-
 .10pt}
    {\vrule height6.3ptdepth-.3ptwidth.11pt\kern-.11pt}
    {\vrule height5.1ptdepth-.30ptwidth.8pt\kern-.8pt}
    {\vrule height4.1ptdepth-.24ptwidth.6pt\kern-.7pt}}
\def\reel{\hbox{I\hskip-2pt R}}

\def\<{\langle}
\def\>{\rangle}

\def\n{{\boldmath n}}

\def\<{\langle}
\def\>{\rangle}

\def\mathbb{\hbox{I\hskip -2pt 1}}
\def\reel{\hbox{I\hskip -2pt R}}

\def\n{{\noindent}}
\newtheorem{theorem}{Theorem}[section]

\newtheorem{proposition}{Proposition}[section]

\newenvironment{proof}{\hspace*{0cm}{\bf Proof}}{\hfill $\Box$
\vspace*{0.5cm}}

\begin{document}

 \vspace{0.5cm}
\title{{\bf
Gaussian representation of a class of Riesz probability
distributions }}
\author{ A. Hassairi \footnote{Corresponding author.
 \textit{E-mail address: abdelhamid.hassairi@fss.rnu.tn}}
\\{\footnotesize{\it
Sfax University Tunisia.}}}

 \date{}
 \maketitle{Running title: \emph{Gaussian representation of a Riesz
 distribution
 }}

\n $\overline{\hspace{15cm}}$\vskip0.3cm  \n {\small {\bf Abstract}}
{\small  : The Wishart probability distribution on symmetric
matrices has been initially defined by mean of the multivariate
Gaussian distribution as an of the chi-square distribution. A more
general definition is given using results for harmonic analysis.
Recently a probability distribution on symmetric matrices called the
Riesz distribution has been defined by its Laplace transform as a
generalization of the Wishart distribution. The aim of the present
paper is to show that some Riesz probability distributions which are
not necessarily Wishart may also be presented by mean of the
Gaussian distribution using Gaussian samples with missing data.
}\\
 \n {\small {\it{ Keywords:}} Chi-square distribution,
Gaussian distribution,
Wishart probability distribution, Riesz probability distribution, Laplace transform.\\

\n AMS Classification : 60B11, 60B15, 60B20.\\
$\overline{\hspace{15cm}}$\vskip1cm

\section{Introduction}
Let $V$ be the space of real symmetric $r\times r$ matrices,
$\Omega$ be the cone of positive definite and $\overline{\Omega}$
the cone of positive semi-definite elements of $V$. We denote the
identity matrix by $e$, the trace of an element $x$ of $V$ by
$tr(x)$ and its determinant by $\Delta(x)$. We equip $V$ by the
scalar product
$$\langle\theta,x \rangle=tr(\theta x).$$ Denoting
$$L_\mu(\theta)=\int_Ve^{\langle\theta,x \rangle}\mu(dx)$$the Laplace transform of a measure $\mu$, it is well known (see
Faraut and Kor\'anyi (1994)) that there exists a positive measure
$\mu_{p}$ on $\overline{\Omega}$ such that the Laplace transform of
$\mu_{p}$ exists for $-\theta\in \Omega$ and is equal to
\begin{equation}\label{R:4}
  \displaystyle\int_{\overline{\Omega}}\exp(tr(\theta x))\mu_{p}(dx)=\Delta^{-p}(-\theta),
\end{equation}if and only if $p$ is in the so called Gindikin set
\begin{equation}\label{R:3}
 \Lambda=\left\{\frac{1}{2},1,\ldots,\frac{r-1}{2}\right\}\cup
  \left(\frac{r-1}{2},+\infty\right).
\end{equation}
 When $p>\frac{r-1}{2}$, the measure $\mu_{p}$ is absolutely continuous with
respect to the Lebesgue measure and is given by
$$\mu_{p}(dx)=\frac{1}{\Gamma_{\Omega}(p)}\Delta^{p-\frac{r+1}{2}}(x)\mathbf{1}_{\Omega}(x)dx,$$
where \\
 $$\Gamma_{\Omega}(p)=(2\pi)^{\frac{n-r}{2}}\prod_{k=1}^{r}\Gamma(p-\frac{k-1}{2}),$$
 where $n$ is the dimension of $V$.\\
And when $p=\frac{j}{2}$ with $j$ integer, $1\leq j\leq r-1$, the
measure
$\mu_{p}$ is singular concentrated on the set of elements of rank $j$ of $\overline{\Omega}$.\\
For $p\in \Lambda$ and $\sigma\in \Omega$, (\ref{R:4}) implies that
the measure $W_{p,\sigma}$ on $\overline{\Omega}$ defined by
\begin{equation}
 W_{p,\sigma}(dx)=\Delta^{-p}(\sigma)\exp(-tr(x\sigma^{-1}))\mu_{p}(dx)
\end{equation}
is a probability distribution. It is called the Wishart distribution
with shape parameter $p$ and scale parameter $\sigma$. Its Laplace
transform exists for $\theta\in \sigma^{-1}-\Omega$ and is equal to
\begin{equation}
\displaystyle\int_{\overline{\Omega}}\exp(tr(\theta
x))W_{p,\sigma}(dx)=\Delta^{-p}(e-\sigma\theta).
\end{equation}
When $p>\frac{r-1}{2}$, the Wishart distribution $W_{p,\sigma}$ is
given by
\begin{equation}
 W_{p,\sigma}(dx)=\frac{\Delta^{-p}(\sigma)}{\Gamma_{\Omega}(p)}\exp(-tr(x\sigma^{-1}))\Delta^{p-\frac{r+1}{2}}(x)\mathbf{1} _{\Omega}(x)dx
\end{equation}
Further details on the Wishart distribution on $\Omega$ may be
obtained from Casalis and Letac (1996) or from Letac and Massam
(1998). It is however important to mention that the Wishart
probability distribution has initially been defined as a
multivariate extension of the chi-square distribution. In fact,
taking $U_{1},. . . , U_{p}$ random vectors in $\reel^{r}$ with
distribution $N(0,\sigma)$, the distribution of the matrix $X
=\sum_{i=1}^{p} U_{i}^{t} U_{i}$ is a Wishart matrix.
\\
This kind of Wishart probability distributions may be defined using
the Gaussian random matrices. Let $u=\ ^{t}(u_{1}, . . . ,u_{r})$ be
a random vector with distribution $N_{r}(0,\sigma)$. Suppose that we
have a random sample of $u$ i.e., a sequence of independent and
identically distributed random variables with distribution equal to
the distribution of $u$, and consider the $r\times s$ Gaussian
matrix
$$U=\left(
  \begin{array}{cccccccccccccccccc}
    u_{1,1}  & . & . & . & u_{1,s}  \\
    .  & . & . & . & .  \\
    .  & . & . & .  & . \\
      . & . & . & . & .  \\
u_{r,1}  & . & . & . & u_{r,s}\\
  \end{array}
\right).$$ The distribution of $U$ is $N_{r,s}(0,\sigma)$ equal to
$$\frac{1}{(2\pi)^{\frac{rs}{2}}\Delta^{\frac{s}{2}}(\sigma)}
e^{-\frac{1}{2}\langle u,\sigma^{-1} u\rangle},$$ where $\langle
u,\sigma^{-1} u\rangle=tr(\sigma^{-1} u\ ^{t}u).$ \\We have
\begin{proposition}\label{3353}If $U\sim N_{r,s}(0,\sigma)$, then
the matrix $X=U^{t}U$ has the Wishart probability distribution
$W_{r}(\frac{s}{2},\frac{\sigma^{-1}}{2})$.
\end{proposition}
\begin{proof}
The Laplace transform of $X$ estimated in
$\theta\in(\frac{\sigma^{-1}}{2}-\Omega)$ is given by
\begin{eqnarray*}
E(e^{\langle\theta,X\rangle})&=&\frac{1}{(2\pi)^{\frac{rs}{2}}\Delta^{\frac{s}{2}}(\sigma)}
\int_{\reel^{rs}}e^{\langle\theta,u
^{t}u\rangle-\frac{1}{2}\langle u,\sigma^{-1} u\rangle}du\\
&=&\frac{1}{(2\pi)^{\frac{rs}{2}}\Delta^{\frac{s}{2}}(\sigma)}
\int_{\reel^{rs}}e^{-\frac{1}{2}\langle u,(\sigma^{-1}-2\theta)u
\rangle}du\\
&=&\frac{\Delta^{\frac{s}{2}}\left((\sigma^{-1}-2\theta)^{-1}\right)}{\Delta^{\frac{s}{2}}(\sigma)}\\
&=&\frac{\Delta^{\frac{s}{2}}\left((\frac{\sigma^{-1}}{2}-\theta)^{-1}\right)}
{\Delta^{\frac{s}{2}}\left((\frac{\sigma^{-1}}{2})^{-1}\right)}.
\end{eqnarray*}
It is then the Laplace transform of a
$W_{r}(\frac{s}{2},\frac{\sigma^{-1}}{2})$ distribution.
\end{proof}\\
This distribution is absolutely continuous with respect to the
Lebesgue measure if and only if $\frac{s}{2}>\frac{r-1}{2}$, that is
if and only if the size $s$ of the sample is greater or equal to the
dimension $r$ of
the space of observations.\\
If we consider $p$ independent random matrices $U_{1},. . . , U_{p}$
with distribution $N_{r,s}(0,\sigma)$, then the distribution of the
matrix $X=\sum_{i=1}^{p} U_{i}\ ^{t} U_{i}\sim W_{r}(\frac{ps}{2},\frac{\sigma^{-1}}{2})$\\
Of course the class of Wishart distributions defined by mean of the
Gaussian vectors or the Gaussian matrices doesn't cover all the
Wishart distributions, however the fact that this random matrix is
expressed in terms of Gaussian random vectors allows to establish
many important properties of this subclass of the class of Wishart
probability distribution.\\
Using a theorem due to Gindikin which relayes on the notion of
generalized power, Hassairi and Lajmi \cite{HaLa} have introduced an
important generalization of the Wishart distribution that they have
called Riesz distribution. This distribution in its general form is
defined by its Laplace transform. We will show that some among the
Riesz distributions may be represented by the Gaussian distribution
using samples of Gaussian vectors with missing data.
\section{Riesz distributions}

\noindent For $ 1\leq i,j\leq r $, we define the matrix
$\mu_{ij}=(\gamma_{k \ell })_{1\leq k,\ell \leq r}$ such that
$\gamma_{ij}=\frac{1}{\sqrt{2}}$ and the other entries are equal to
 0. We also set $$c_i=\sqrt{2}\mu_{ii} \textrm{ for } 1\leq i\leq r,
 \textrm{ and } e_{ij}= (\mu_{ij}+\mu_{ji}), \textrm{ for }  1\leq i< j\leq
 r.$$
 With these notations, an element $x$ of $V$ may be written $$x=\displaystyle\sum_{i=1}^rx_ic_i+\sum_{i<j}x_{ij} e_{ij}.$$
In particular, for $1\leq k\leq r, $ we set
 $e_k=c_{1}+\cdot\cdot\cdot+c_{k}$.\\
Now consider the map $P_k$ from $V$ into $V$ defined by
$$x=\displaystyle\sum_{i=1}^rx_ic_i+\sum_{i<j}x_{ij} e_{ij} \longmapsto P_k(x)=\displaystyle\sum_{i=1}^kx_ic_i+\sum_{i<j\leq k}x_{ij} e_{ij}.$$
Then the determinant $\Delta^{(k)}(P_k(x))$ of the $k\times k$ bloc
$P_k(x)$ is denoted by $\Delta_k(x)$, it is the principal minor of
$x$ of order $k$. The generalized power of an element $x$ of
$\Omega$ is then defined for $s=(s_1,\cdot\cdot\cdot,s_r)\in
\reel^r$, by
$$\Delta_s(x)=\Delta_1(x)^{s_1-s_2}\Delta_2(x)^{s_2-s_3}\cdot\cdot\cdot\Delta_r(x)^{s_r}.$$
Note that $\Delta_s(x)=\Delta^p( x)$ if $s=(p,\cdot\cdot\cdot,p)$
with $p \in \reel$. It is also easy to see that
$\Delta_{s+s'}(x)=\Delta_s(x)\Delta_{s'}(x)$. In particular, if $m
\in \reel \ $ and $s+m=(s_1+m,\cdot\cdot\cdot,s_r+m)$, we have
$\Delta_{s+m}(x)=\Delta_{s}(x)\Delta^m(x)$. \\
Now, we introduce the set $\Xi$ of elements $s$ of $\reel^r$ defined
as follows:\\ For a real $u\geq 0$, we put $\varepsilon(u)=0$ if
$u=0$ and $\varepsilon(u)=1$ if $u>0$. \\For $u_1,u_2,\cdots,u_r\geq
0,$ we define $s_1=u_1$ and
$s_k=u_k+\frac{1}{2}(\varepsilon(u_1)+\cdots+\varepsilon(u_{k-1}))$,
for $2\leq k\leq r$.\\
We have the following result due to Gindikin \cite{G}, it is also
proved in the monograph by Faraut and Kor\'{a}nyi \cite{Farau}.
\begin{theorem}
There exists a positive measure $R_s$ on $V$ such that the Laplace
transform $L_{R_s}$ is defined on $-\Omega$ and is equal to
$\Delta_s(-\theta^{-1})$ if and only if $s$ is in $\Xi$.
\end{theorem}
The measures $R_s$, defined in the previous theorem in terms of
their Laplace transforms are divided into two classes according to
the position of $s$ in $\Xi$. In the first class, the measures are
absolutely continuous with respect to the Lebesgue measure on
$\Omega$. More precisely, we have (see \cite{HaLa}) that when
$s=(s_1,\cdots,s_r)$ in $\Xi$ is such that for all $i$,
$s_i>\frac{i-1}{2}$, $$R_s(dx)=\frac{1}{\Gamma _{\Omega }(s)}\Delta
_{s-\frac{r+1}{2}}(x) {\mathbf{1}}_{\Omega }(x)dx$$ where
$\Gamma_{\Omega }(s)=(2\pi
)^{\frac{n-r}{2}}\displaystyle\prod_{j=1}^r\Gamma
(s_{j}-\frac{j-1}{2})$. \\The second class corresponds to $s$ in
$\Xi \setminus \prod_{i=1}^{r}]\frac{i-1}{2},+\infty[$. In this case
the Riesz measure $R_s$ is concentrated on the boundary
$\partial\Omega$ of $\Omega$, it has a complicated form which is
explicitly described in \cite{HaLa1}.\\ For $s=(s_{1},\cdots,s_{r})$
in $\Xi$ and $\sigma$ in $\Omega$, we define the Riesz distribution
$R_r(s,\sigma )$ by
$$R_r(s,\sigma )(dx)=\frac{e^{-\langle\sigma ,x\rangle}}{\Delta _{s}(\sigma
^{-1})}R_s(dx).$$ The Laplace transform of the Riesz distribution is
defined for $\theta$ in $\sigma-\Omega$ by
\begin{equation}\label{dd}
L_{R_r(s,\sigma
)}(\theta)=\frac{\Delta_{s}((\sigma-\theta)^{-1})}{\Delta
_{s}(\sigma ^{-1})}
\end{equation}
If $s=(p,\cdots,p)$ such that $p\in
\Lambda=\{0,\frac{1}{2},\cdots,\frac{r-1}{2}\}\cup
]\frac{r-1}{2},+\infty[$, then $R_r(s,\sigma)$ is nothing but the
Wishart distribution with shape
parameter $p$ and scale parameter $\sigma$.\\
The definition of the absolutely continuous Riesz distribution on
the cone of positive definite symmetric matrices $\Omega$ relies on
the notion of generalized power. It is defined for $\sigma$ in
$\Omega$ and $s=(s_1,\cdot\cdot\cdot,s_r)\in\reel^r$ such that
$s_{i}>\frac{i-1}{2}$ by

$$
R(s,\sigma )(dx)=\frac{1}{\Gamma _{\Omega }(s)\Delta _{s}(\sigma
^{-1})}e^{-\langle\sigma ,x\rangle}\Delta _{s-\frac{r+1}{2}}(x)
{\mathbf{1}}_{\Omega }(x)dx.
$$

\section{Main result}
We now show that some of the Riesz distributions have a
representation in terms of the Gaussian distribution generalizing
what accours in the case of the Wishart distribution.
For simplicity, we take $\sigma$ equal to the identity matrix.\\
Suppose that we have $s_{r}$ independent observations of the vector
$(u_{1}, . . . ,u_{r})$ with Gaussian distribution $N_{r}(0,I_{r})$,
and that some data in the observations are missing. The sample is
then partitioned according to the number of available components in
the observation. More precisely, we suppose that there exist
integers $0<s_{1}\leq s_{2}\leq...\leq s_{r}$ such that we have

 $s_{1}$ observations on $(u_{1}, . . . ,u_{r})$ where all the components of the
 observation are available,

 $(s_{2}-s_{1})$ observations on $(u_{2}, . . . ,u_{r})$, the first component of the
 observation is missing,

$(s_{i+1}-s_{i})$ observations on $(u_{i+1}, . . . ,u_{r})$, the $i$
first components of the
 observation are missing, $2\leq i\leq r-1$.\\
 The number $s_{i}$ is in particular equal to the number of times the component $u_{i}$ appears in the data.\\
  Define

 $p_{1}=\sup\{i\geq1; s_{i}=s_{1}\ \},$

 $p_{2}=\sup\{i>p_{1}; s_{i}= s_{p_{1}+1}\ \},$

 $p_{l+1}=\sup\{i>p_{l}; s_{i}= s_{p_{l}+1} \},$

 $p_{k}=\inf\{i; s_{i}=s_{r} \}-1,$\\
 so that we have

 $s_{1}=...=s_{p_{1}}$,

 $s_{p_{1}}\neq s_{p_{1}+1}$,

 $s_{p_{1}+1}=...=s_{p_{2}}$,

  $s_{p_{2}}\neq s_{p_{2}+1}$

 $ \ \ \ .$

 $ \ \ \ .$

 $ \ \ \ .$

 $s_{p_{l}+1}=...=s_{p_{l+1}}$,

 $ \ \ \ .$

 $ \ \ \ .$

 $ \ \ \ .$

 $s_{p_{k}}\neq s_{r}$

 $s_{p_{k}+1}=...=s_{r}$.\\
We have that $p_{k+1}=r$ .\\Replacing the missing components by the
theoretical mean that is by zero, we obtain the matrix
\begin{equation}\label{3454}
U=\left(
  \begin{array}{cccccccccccccccccc}
    u_{1,1}  & . & . & u_{1,s_{p_{1}}} & 0 & . & . & . & . & . & .  & 0 \\
    .  & . & . & . & . & . & . & . & . & . & . & .  \\
    .  & . & . & .  & . & . & . & . & . & . & . & .  \\
     u_{p_{1},1}  & . & . & u_{p_{1},s_{p_{1}}}& 0 &. & . & . & . & .& . &  0\\
    u_{p_{1}+1,1}  & . & . & u_{p_{1}+1,s_{p_{1}}} & . & . &u_{p_{1}+1,s_{p_{2}}} & 0   & . & . & . &  0 \\
     . & . & . & . & . & . & . & . & . & . & . & . \\
      . & . & . & . & . & . & . & . & . & . & . & .   \\
    u_{p_{2},1}  & . & . & u_{p_{2}-1,s_{p_{1}}} & . & . &u_{p_{2},s_{p_{2}}} & 0 &  . & . & . & 0 \\
     . & . & . & . & . & . & . & . & . & . &  . & .  \\
      . & . & . & . & . & . & . & . & . & . &  . & .  \\
      . & . & . & . & . & . & . & . & . & . & . & .  \\
     u_{p_{k}+1,1}  & . & . &. & . & . &.& .&  . & . & .  & u_{p_{k}+1,s_{r}} \\
      . & . & . & . & . & . & . & . & . & . &  . & .  \\
      . & . & . & . & . & . & . & . & . & . &  .& . \\
u_{r,1}  & . & . & . & . & . & . & . & . & . & .  & u_{r,s_{r}}\\
  \end{array}
\right)
\end{equation}
This matrix $U$ consists of blocks defined recursively in the
following way :

$$U^{(1)}=\left(u_{i,j}\right)_{1\leq i \leq p_{1} \\
1\leq j \leq s_{p_{1}}}$$
\begin{equation}\label{3465}
U^{(l+1)}=\left(
  \begin{array}{cccccccccccccccccc}
   U^{(l+1)}_{1}  & 0 \\
    U^{(l+1)}_{21} & U^{(l+1)}_{0}
\end{array}
\right),\ 1\leq l \leq k
\end{equation}
with

$U^{(l+1)}_{1}=U^{(l)}$

 $U^{(l+1)}_{21}= \left(u_{i,j}\right)_{p_{l}+1\leq i \leq p_{l+1} \ , \ 1\leq j \leq s_{p_{l}}}$

  $U^{(l+1)}_{0}= \left(u_{i,j}\right)_{p_{l}+1\leq i \leq p_{l+1} \ , \ s_{p_{l}}+1\leq j \leq
  s_{p_{l+1}}}$.\\
 We have that $U^{(k+1)}=U$.

\begin{theorem}\label{3499}
 The distribution of  $X=U^{t}U$ is $R(s,\frac{I_{r}}{2})$
with $s=(\frac{s_{1}}{2},...,\frac{s_{r}}{2})$.
\end{theorem}
\begin{proof}
We use the following Peirse decomposition of an $r\times r$ matrix
$\omega$. \\For $1\leq l \leq k$, we set
$$\omega^{(l)}=\left(\omega_{i,j}\right)_{1\leq i \leq p_{l} \\
1\leq j \leq p_{l}},$$ and we write $\omega^{(l+1)}$ as
\begin{equation}\label{3508}
\omega^{(l+1)}=\left(
  \begin{array}{cccccccccccccccccc}
    \omega^{(l)}& ^{t}\omega^{(l+1)}_{21}\\
     \omega^{(l+1)}_{21} & \omega^{(l+1)}_{0}
\end{array}
\right).
\end{equation}
 From Proposition \ref{3353}, we have that $U^{(1)} \
^{t}U^{(1)}$ has the
$W_{p_{1}}(\frac{s_{p_{1}}}{2},\frac{I_{p_{1}}}{2})$ distribution.
On the other hand, since  $s_{1}=...=s_{p_{1}}$, we have that
$$W_{p_{1}}\left(\frac{s_{p_{1}}}{2},\frac{I_{p_{1}}}{2}\right)=
R_{p_{1}}\left(\left(\frac{s_{1}}{2},...,\frac{s_{p_{1}}}{2}\right),\frac{I_{p_{1}}}{2}\right).$$
It suffises to show that if $U^{(l)}\ ^{t}U^{(l)}$ is
$R_{p_{l}}\left(\left(\frac{s_{1}}{2},...,\frac{s_{p_{l}}}{2}\right),\frac{I_{p_{l}}}{2}\right)$,
then $U^{(l+1)}\ ^{t}U^{(l+1)}$ is\\
$R_{p_{l+1}}\left(\left(\frac{s_{1}}{2},...,\frac{s_{p_{l+1}}}{2}\right),\frac{I_{p_{l+1}}}{2}\right)$.
\\Again we use the Laplace transform. For
$\theta\in(\frac{I_{r}}{2}-\Omega)$ and with the decomposition
(\ref{3508}), we have
\begin{eqnarray*}
L_{U^{(l+1)}\ ^{t}U^{(l+1)}}(\theta^{(l+1)})
&=&E\left(\exp\left(\langle\theta^{(l+1)},U^{(l+1)}\
^{t}U^{(l+1)}\rangle\right)\right)\\
&=&E\left(\exp\left(tr\left(\theta^{(l+1)}_{0}U^{(l+1)}_{0}\ ^{t}U^{(l+1)}_{0}\right)\right)\right)\\
&& E\left(\exp \left(tr\left(\theta^{(l)}U^{(l)}\ ^{t}U^{(l)}+ \
^{t}\theta^{(l+1)}_{21}U^{(l+1)}_{21}\
^{t}U^{(l)}\right.\right.\right.
\\ && +\left.\left.\left.U^{(l)}\
^{t}U^{(l+1)}_{21}\theta^{(l+1)}_{21}+\theta^{(l+1)}_{0}U^{(l+1)}_{21}\
^{t}U^{(l+1)}_{21} \right)\right)\right).
\end{eqnarray*}
We have that $U^{(l+1)}_{0}$ is a $(p_{(l+1)}-p_{(l)})\times
(s_{p_{(l+1)}}-s_{p_{(l)}})$ Gaussian matrix with distribution
$N(0,I_{(p_{(l+1)}-p_{(l)}})$. \\According to Proposition
\ref{3353}, the matrix $U^{(l+1)}_{0}\ ^{t}U^{(l+1)}_{0}$ has the
Wishart distribution
$W_{r}(\frac{s_{p_{(l+1)}}-s_{p_{(l)}}}{2},\frac{I_{(p_{(l+1)}-p_{(l)})}}{2})$.
Hence
\begin{eqnarray*}
E\left(\exp\left(tr\left(\theta^{(l+1)}_{0}U^{(l+1)}_{0}\
^{t}U^{(l+1)}_{0}\right)\right)\right)
=\frac{\Delta^{\frac{s_{p_{(l+1)}}-s_{p_{(l)}}}{2}}\left((\frac{I_{(p_{(l+1)}-p_{(l)})}}{2}-\theta^{(l+1)}_{0})^{-1}\right)}
{\Delta^{\frac{s_{p_{(l+1)}}-s_{p_{(l)}}}{2}}\left(2(I_{(p_{(l+1)}-p_{(l)})})\right)}.
\end{eqnarray*}
Concerning the second expectation, we have that
\begin{eqnarray*}
&&tr\left( \ ^{t}\theta^{(l+1)}_{21}U^{(l+1)}_{21}\ ^{t}U^{(l)}
+U^{(l)}\
^{t}U^{(l+1)}_{21}\theta^{(l+1)}_{21}+\theta^{(l+1)}_{0}U^{(l+1)}_{21}\
^{t}U^{(l+1)}_{21} \right)\\ &&= 2\langle U^{(l+1)}_{21} ,\
\theta^{(l+1)}_{21} U^{(l)}\rangle +\langle U^{(l+1)}_{21} ,\
\theta^{(l+1)}_{0}U^{(l+1)}_{21}\rangle
\end{eqnarray*}
As $U^{(l+1)}_{21}$ is a $(p_{(l+1)}-p_{(l)})\times s_{p_{l}}$
matrix with distribution $N(0,I_{(p_{(l+1)}-p_{(l)})})$, we obtain
that
\begin{eqnarray*}
&&E\left(\exp\left(2\langle U^{(l+1)}_{21} ,\ \theta^{(l+1)}_{21}
U^{(l)}\rangle +\langle U^{(l+1)}_{21} ,\
\theta^{(l+1)}_{0}U^{(l+1)}_{21}\rangle\right)|U^{(l)}\right)
\\&&=\Delta^{\frac{ s_{p_{(l)}}}{2}}
\left(\left(I_{(p_{(l+1)}-p_{(l)}}-2\theta^{(l+1)}_{0}\right)^{-1}\right)
\\&&\ \ \ \ \exp\left(2\langle \theta^{(l+1)}_{21}
U^{(l)},\left(I_{(p_{(l+1)}-p_{(l)}}-2\theta^{(l+1)}_{0}\right)^{-1}
\theta^{(l+1)}_{21} U^{(l)}\rangle\right).
\end{eqnarray*}
Now multiplying this by $\exp \left(tr\left(\theta^{(l)}U^{(l)}\
^{t}U^{(l)}\right)\right)$, we need then to calculate the
expectation
$$E\left(\exp\left(\left\langle
\left(\theta^{(l)}+^{t}\theta^{(l+1)}_{21}
\left(\frac{I_{(p_{(l+1)}-p_{(l)}}}{2}-\theta^{(l+1)}_{0}\right)^{-1}
\theta^{(l+1)}_{21} \right),U^{(l)}\
^{t}U^{(l)}\right\rangle\right)\right).$$ As from the hypothesis,
$U^{(l)}\ ^{t}U^{(l)}$ is
$R_{p_{l}}\left(\left(\frac{s_{1}}{2},...,\frac{s_{p_{l}}}{2}\right),\frac{I_{p_{l}}}{2}\right)$,
using (\ref{dd}) this is equal to
$$\frac{\Delta_{\left(\frac{s_{1}}{2},...,\frac{s_{p_{l}}}{2}\right)}
\left(\left(\frac{I_{p_{l}}}{2}-\left(\theta^{(l)}+^{t}\theta^{(l+1)}_{21}
\left(\frac{I_{(p_{(l+1)}-p_{(l)}}}{2}-\theta^{(l+1)}_{0}\right)^{-1}
\theta^{(l+1)}_{21} \right)\right)^{-1}\right)}{\Delta
_{\left(\frac{s_{1}}{2},...,\frac{s_{p_{l}}}{2}\right)}(2I_{p_{l}})}.
$$
Finally, we obtain
\begin{eqnarray*}
&&L_{U^{(l+1)}\ ^{t}U^{(l+1)}}(\theta^{(l+1)})
=\frac{\Delta^{\frac{s_{p_{(l+1)}}-s_{p_{(l)}}}{2}}\left(\left(\frac{I_{(p_{(l+1)}-p_{(l)}}}{2}-\theta^{(l+1)}_{0}\right)^{-1}\right)}
{\Delta^{\frac{s_{p_{(l+1)}}-s_{p_{(l)}}}{2}}\left(2(I_{(p_{(l+1)}-p_{(l)})})\right)}\\&&
\frac{\Delta^{\frac{ s_{p_{(l)}}}{2}}
\left(\left(\frac{I_{(p_{(l+1)}-p_{(l)}}}{2}-\theta^{(l+1)}_{0}\right)^{-1}\right)}
{\Delta^{\frac{s_{p_{(l)}}}{2}}\left(2(I_{(p_{(l+1)}-p_{(l)})})\right)}\\&&\frac{\Delta_{\left(\frac{s_{1}}{2},...,\frac{s_{p_{l}}}{2}\right)}
\left(\left(\frac{I_{p_{l}}}{2}-\theta^{(l)}\
-^{t}\theta^{(l+1)}_{21}
\left(\frac{I_{(p_{(l+1)}-p_{(l)}}}{2}-\theta^{(l+1)}_{0}\right)^{-1}
\theta^{(l+1)}_{21} \right)^{-1}\right)}{\Delta
_{\left(\frac{s_{1}}{2},...,\frac{s_{p_{l}}}{2}\right)}(2I_{p_{l}})},
\end{eqnarray*}
that is
\begin{eqnarray*}
&&L_{U^{(l+1)}\ ^{t}U^{(l+1)}}(\theta^{(l+1)}) =
\frac{\Delta^{\frac{ s_{p_{(l+1)}}}{2}}
\left(\left(\frac{I_{(p_{(l+1)}-p_{(l)}}}{2}-\theta^{(l+1)}_{0}\right)^{-1}\right)}
{\Delta^{\frac{s_{p_{(l+1)}}}{2}}\left(2(I_{(p_{(l+1)}-p_{(l)})})\right)}\\&&\frac{\Delta_{\left(\frac{s_{1}}{2},...,\frac{s_{p_{l}}}{2}\right)}
\left(\left(\frac{I_{p_{l}}}{2}-\theta^{(l)}\
-^{t}\theta^{(l+1)}_{21}
\left(\frac{I_{(p_{(l+1)}-p_{(l)}}}{2}-\theta^{(l+1)}_{0}\right)^{-1}
\theta^{(l+1)}_{21} \right)^{-1}\right)}{\Delta
_{\left(\frac{s_{1}}{2},...,\frac{s_{p_{l}}}{2}\right)}(2I_{p_{l}})}.
\end{eqnarray*}
But  $s_{p_{l}+1}=...=s_{p_{l+1}}$, then
$$\Delta^{\frac{s_{p_{(l+1)}}}{2}}\left(2(I_{(p_{(l+1)}-p_{(l)})})\right)=
\Delta_{\left(\frac{s_{p_{l}+1}}{2},...,\frac{s_{p_{l+1}}}{2}\right)}\left(2(I_{(p_{(l+1)}-p_{(l)})})\right),$$
and
$$\Delta^{\frac{
s_{p_{(l+1)}}}{2}}
\left(\left(\frac{I_{(p_{(l+1)}-p_{(l)}}}{2}-\theta^{(l+1)}_{0}\right)^{-1}\right)=
\Delta_{\left(\frac{s_{p_{l}+1}}{2},...,\frac{s_{p_{l+1}}}{2}\right)}
\left(\left(\frac{I_{(p_{(l+1)}-p_{(l)}}}{2}-\theta^{(l+1)}_{0}\right)^{-1}\right).$$
As we have
$$\Delta
_{\left(\frac{s_{1}}{2},...,\frac{s_{p_{l}}}{2}\right)}(2I_{p_{l}})
\Delta_{\left(\frac{s_{p_{l}+1}}{2},...,\frac{s_{p_{l+1}}}{2}\right)}\left(2(I_{(p_{(l+1)}-p_{(l)})})\right)=
\Delta_{\left(\frac{s_{1}}{2},...,\frac{s_{p_{l+1}}}{2}\right)}\left(2(I_{(p_{(l+1)})})\right),
$$
it follows that
$$L_{U^{(l+1)}\ ^{t}U^{(l+1)}}(\theta^{(l+1)})=
\frac{\Delta_{\left(\frac{s_{1}}{2},...,\frac{s_{p_{l+1}}}{2}\right)}
\left(\left(\frac{I_{p_{(l+1)}}}{2}-\theta^{(l+1)}\right)^{-1}\right)}
{\Delta_{\left(\frac{s_{1}}{2},...,\frac{s_{p_{l+1}}}{2}\right)}\left(2(I_{(p_{(l+1)})})\right)}.$$
Therefore $U^{(l+1)}\ ^{t}U^{(l+1)}$ is a Riesz random matrix with
distribution
$$R_{p_{l+1}}\left(\left(\frac{s_{1}}{2},...,\frac{s_{p_{l+1}}}{2}\right),\frac{I_{p_{l+1}}}{2}\right).$$
Given that $X=U^{(k+1)}\ ^{t}U^{(k+1)}$, the result follows.
\end{proof}\\

\end{document}